\documentclass[letterpaper, 10 pt, conference]{ieeeconf}  

\pdfoutput=1
\IEEEoverridecommandlockouts                              

\usepackage{eqnarray}
\usepackage{graphics} 
\usepackage{epsfig} 
\usepackage{epstopdf}
\usepackage{amsmath} 
\usepackage{amssymb}  
\usepackage{amsthm}
\usepackage{xcolor}
\usepackage{soul}
\usepackage{subcaption}
\usepackage{mwe}
\usepackage{array}
\usepackage{romannum}
\usepackage{multirow}
\usepackage{balance}
\usepackage{mathtools}
\usepackage{bm}

\usepackage{hyperref}
\newcolumntype{L}[1]{>{\raggedright\let\newline\\\arraybackslash\hspace{0pt}}m{#1}}
\newcolumntype{C}[1]{>{\centering\let\newline\\\arraybackslash\hspace{0pt}}m{#1}}
\newcolumntype{R}[1]{>{\raggedleft\let\newline\\\arraybackslash\hspace{0pt}}m{#1}}

\theoremstyle{plain}
\newtheorem{thm}{Theorem}
\newtheorem{lemma}{Lemma}
\newtheorem{cor}{Corollary}

\theoremstyle{remark}
\newtheorem{remark}{Remark}

\theoremstyle{remark}

\theoremstyle{definition}
\newtheorem{defn}[thm]{Definition} 
\usepackage{calrsfs}
\DeclareMathAlphabet{\pazocal}{OMS}{zplm}{m}{n}

\newcommand{\RN}[1]{%
  \textup{\uppercase\expandafter{\romannumeral#1}}%
}

\newcommand{\mf}[1]{\mathbf{#1}}

\title{\LARGE \bf
Differential-Flatness and Control of Quadrotor(s) with a Payload Suspended through Flexible Cable(s)}

\author{Prasanth Kotaru$^{1}$, Guofan Wu$^{2}$ and Koushil Sreenath$^{1}$
\thanks{*This work is supported in part by NSF Grants IIS-1464337 and CMMI-1538869, PITA, Autel and in part by the Google faculty research award.}
\thanks{$^{1}$P. Kotaru and K. Sreenath are with the Dept. of Mechanical Engineering, University of California, Berkeley, CA, 94720, 
        {\tt\small \{prasanth.kotaru, koushils\}@berkeley.edu}}%
\thanks{$^{2}$ G. Wu is with Department of Mechanical Engineering,
Carnegie Mellon University, 5000 Forbes Avenue, 	Pittsburgh PA, 15213, 
        {\tt\small gwu@andrew.cmu.edu}}%
}

\begin{document}
\setlength{\abovedisplayskip}{0pt}
\setlength{\belowdisplayskip}{0pt}
\setlength{\textfloatsep}{2pt}

\maketitle
\thispagestyle{empty}
\pagestyle{empty}

\begin{abstract}
We present the coordinate-free dynamics of three different quadrotor systems 
: (a) single quadrotor with a point-mass payload suspended through a flexible cable;
(b) multiple quadrotors with a shared point-mass payload suspended through flexible cables; and (c) multiple quadrotors with a shared rigid-body payload suspended through flexible cables.
We model the flexible cable(s) as a finite series of links with spherical joints with mass concentrated at the end of each link. The resulting systems are thus high-dimensional with high degree-of-underactuation.
For each of these systems, we show that the dynamics are differentially-flat, enabling planning of dynamically feasible trajectories.
For the single quadrotor with a point-mass payload suspended through a flexible cable with five links (16 degrees-of-freedom and 12 degrees-of-underactuation), we use the coordinate-free dynamics
to develop a geometric variation-based linearized equations of motion about a desired trajectory. We show that a finite-horizon linear quadratic regulator can be used to track a desired trajectory with a relatively large region of attraction.
%
%
%
%

\end{abstract}

\section{INTRODUCTION}

Aerial transportation through small unmanned aerial vehicles (UAVs) has shown great potential in recent years, especially with the commercialization of UAV-based package and mail delivery.
Consequently, the automatic control of quadrotors to transport payloads has been the focus for many research groups. 
Load carrying using quadrotor UAVs can be realized either by rigidly attaching the load to the quadrotor or suspending the load through cables. 
A rigidly attached load can increase the inertia of the quadrotor, making it sluggish for fast attitude response and agile disturbance rejection.
A cable-suspended load system increases the degrees of underactuation, making planning and control for such systems more challenging. 

Control of UAVs with a suspended load have been addressed through trajectory generation for fast load transport with minimized swing \cite{palunko2012trajectory}, \cite{zameroski2008rapid}, 
or through modeling the suspended load as an external disturbance and developing robust controllers to reject these disturbances \cite{pizetta2015modelling}. Geometric control design has been developed in \cite{sreenath2013geometric}, \cite{sreenath2013trajectory} to track a smooth aggressive trajectory.
Similar geometric controllers have been proposed in \cite{wu2014geometric}, which allows the load to undergo large swings. Similar controllers for suspended loads have been developed in \cite{sreenathRSS2013}, \cite{lee2013geometric}, \cite{wu2014geometric}, where the load is supported from multiple quadrotors.
\begin{figure}[t]
\centering
\includegraphics[width=\columnwidth]{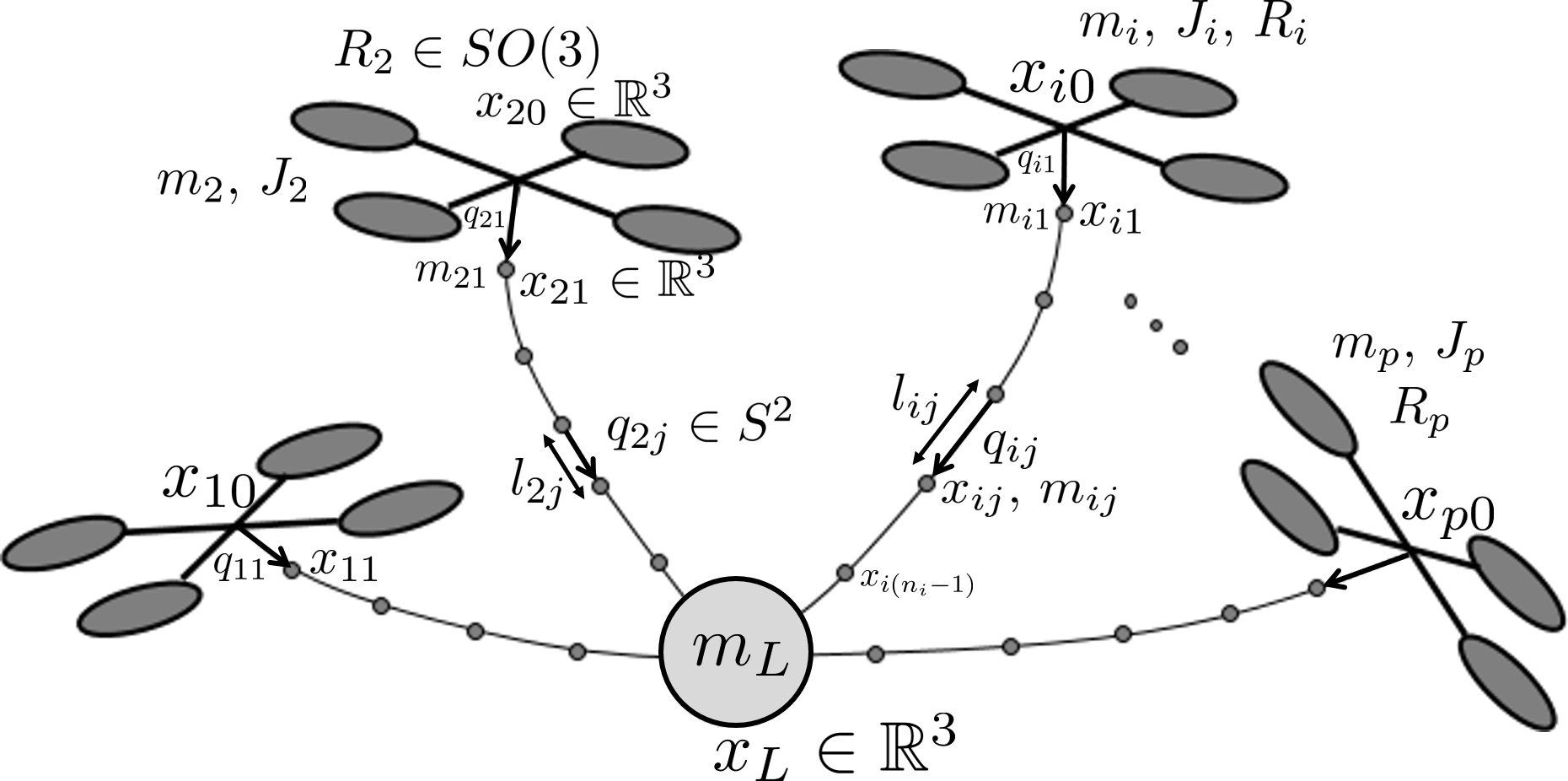}
\caption{Multiple quadrotors with a shared point-mass payload supsended through flexible cables.}
\label{fig:quad3}
\end{figure}

However, these controllers assume that the suspended cable is massless and that the cable is always taut. In particular, they do not address the control challenges when the cable is not taut or when the cable is deformed. These assumptions may not hold in reality, especially when the mass of the cable is comparable relative to the suspended load and/or is distributed or in cases where the tension in the cable is very small.
In this case, the stability of these controllers would get worse, and thus the mass distribution of the cable needs to be considered in the dynamics.
However, a continuous mass distribution would result in a configuration space of infinite dimension with the dynamics being represented through partial differential equations. 
To reduce the modeling complexity, a general methodology is to employ a finite element approximation for the cable, where the cable is approximated as a series of links connected by spherical joints \cite{goodarzi2014geometric,goodarzi2015geometric,lee2015geometric}.
Goodarzi et al. \cite{goodarzi2014geometric} first develops dynamics of a single quadrotor transporting a point-mass through a flexible cable.
Based on this, \cite{goodarzi2015dynamics} extends to the case of a rigid body load with multiple quadrotors.
Although both these work present a coordinate-free model and use linearization for regulation control, the resulting controller can only stabilize to a setpoint corresponding to the quadrotor hovering and the payload suspended vertically.  In particular, the developed controller is unable to track a desired trajectory.

In this paper, we focus on the properties of three particular transportation systems with flexible cables. With respect to the prior work in \cite{goodarzi2014geometric,goodarzi2015geometric} which proposes regulation control of the load's pose, our aim is to  investigate further into the planning and tracking of desired dynamically feasible trajectories for such systems.  The contributions of this paper with respect to prior work is as follows:
\begin{itemize}
\item We develop coordinate-free dynamics of three different quadrotor systems with a payload suspended through flexible cables using the Newton-Euler method.  We prove that the resulting dynamics for these systems are differentially-flat and provide flat outputs.
\item For the single quadrotor with point-mass payload suspended through a flexible cable, we present a geometric variation-based linearization of the system dynamics with respect to a desired reference trajectory. 
\item We use the linearized dynamics to develop a finite-horizon linear quadratic regulator and demonstrate trajectory tracking on the nonlinear system to achieve trajectory tracking of a sufficiently smooth reference trajectory of the load.  We demonstrate the large region of attraction of the controller through numerical simulations. 
\end{itemize}

The rest of the paper is organized as follows. Section \ref{sec:dynamics} presents the dynamical models and assumptions of the described three quadrotor systems with flexible cables. Section \ref{sec:diff-flat} demonstrates that these systems are differentially flat. Section \ref{sec:control} develops the linearized dynamics and presents a finite-horizon linear quadratic regulator to achieve the trajectory tracking for load suspended form a quadrotor. Section \ref{sec:simulation} presents the simulation results for load trajectory tracking, and Section \ref{sec:conclusion} provides concluding remarks.

\section{SYSTEM DYNAMICS}
\label{sec:dynamics}
In this section, we present the dynamical models for a single or multiple quadrotor systems with payload suspended through flexible cables, where an individual flexible cable is modeled as a series of $n$ small links as illustrated in Fig.~\ref{fig:quad3}$-$Fig.~\ref{fig:quad4}. 
We describe the coordinate-free dynamics for these systems using a rotation matrix in $SO(3):=\{R\in \mathbb{R}^{3\times 3}| R^TR=I, det(R) = +1 \}$ for quadrotor attitude, and a unit-vector in the two-sphere $S^2:=\{q\in \mathbb{R}^3|q.q=1\}$ for each of the $n$ links of the cable.

The configuration of the systems under consideration can be given by the pose of the load in inertial frame (position for point-mass load \& position and orientation for rigid-body load), attitudes of each link in the flexible cable(s) and the attitude of the quadrotor(s). Equations of motion are presented in Newton-Euler method which makes it convenient in the later sections.

\subsection{Quadrotor with load suspended through a flexible cable (Fig.~\ref{fig:quad1})}

\begin{figure}[t]
\centering
\includegraphics[scale=0.3]{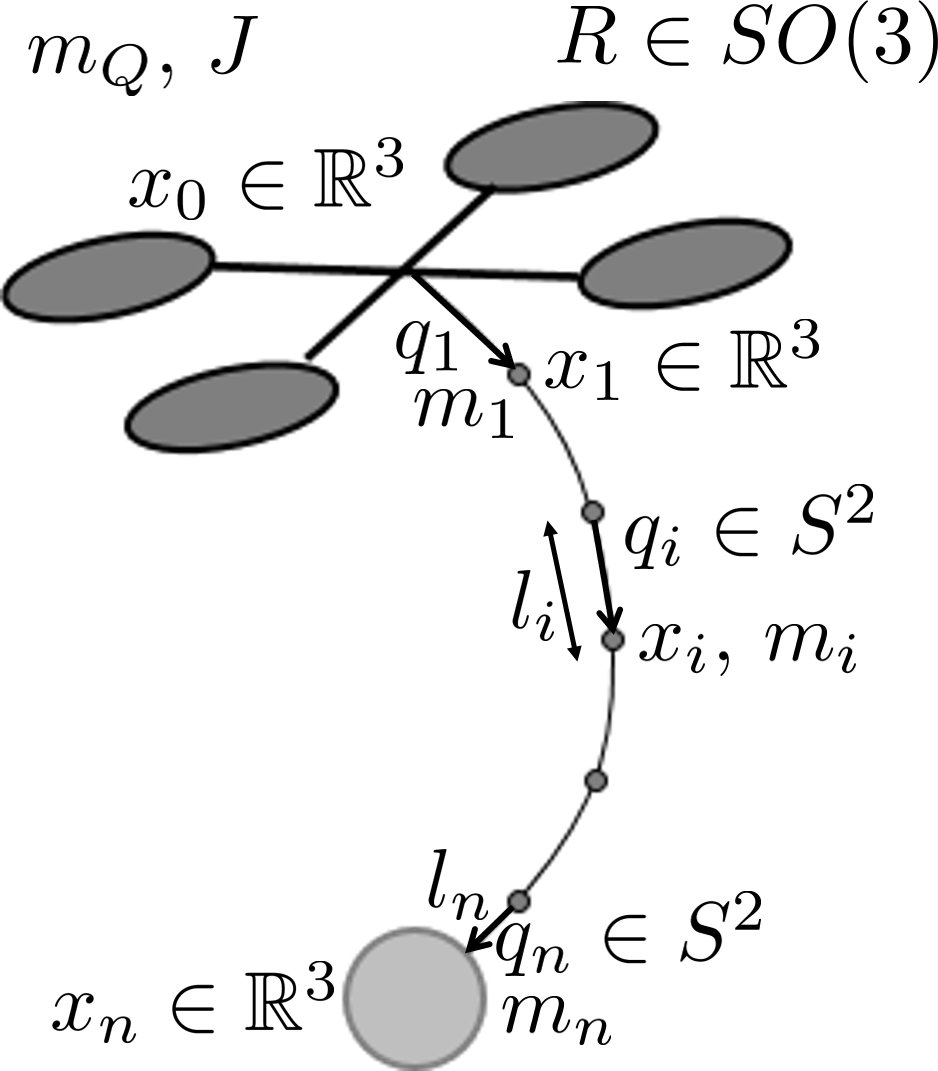}
\caption{Quadrotor with a point-mass payload suspended through a flexible cable.  The flexible cable is modeled as a series of links connected by $S^2$ joints. The system evolves on $SO(3)\times \mathbb{R}^3 \times (S^2)^n$ and  has $(6+2n)$ degrees of freedom with $(2+2n)$ degrees of underactuation.}	
\label{fig:quad1}
\end{figure}

The first system is a single quadrotor with load suspended through a flexible cable shown in Fig.~\ref{fig:quad1}. 
The flexible cable is modeled as a chain of $n$ links, and the suspended load is considered to be a point mass at the end of the $n^{th}$ link. 
The configuration space is given by $Q \coloneqq SO(3)\times\mathbb{R}^3\times(S^2)^n$. 
In the finite element approximation, we assume that the mass of each link is concentrated at the end of the link. 
The relation between the different link-mass positions, quadrotor center-of-mass and load is given as 
\begin{equation}
\label{eq:dynA}
x_i = x_{i-1} + l_iq_i, 
\end{equation}
where $i\in \{1,2,\hdots,n\}$ and $l_i$, $q_i$, $x_i$ are respectively the length, the unit directional vector and the position of the $i^{th}$ link; $x_0$ and $x_n$ are the positions of quadrotor center-of-mass and load; $R$ and $\Omega$ are the rotation matrix of the quadrotor and its body-fixed angular velocity.
Let $T_i \in \mathbb{R}$ be the magnitude of the tension in the $i^{th}$ link, the dynamics of the system can be written as follows,
\begin{align}
\label{eq:diff1}
m_Q(\ddot{x}_0 + ge_3) =&\ fRe_3 + T_1q_1,
\\
\label{eq:diff2}
m_j(\ddot{x}_j + ge_3) =&\ - T_jq_j + T_{(j+1)}q_{(j+1)},
\\
\label{eq:diff3}
m_n(\ddot{x}_n + ge_3) =&\ -T_nq_n,
\\
\label{eq:diff4}
J\dot{\Omega} + \hat{\Omega}J\Omega =&\ M,
\end{align}
$\forall$ $ j \in \{ 1,2,\hdots,(n-1)\}$ where $m_Q$, $J$, $f$ and $M$ are the mass, inertia matrix, thrust and the moment represented in body frame of the quadrotor, and $m_i$ is the mass of the $i^{th}$ link. (The hat-map $\hat{\cdot}:\mathbb{R}^3\rightarrow so(3)$ is defined, as, $\hat{x}y = x\times y$ for any $x,y\in \mathbb{R}^3$).	

\begin{remark}
This system has $6+2n$ degrees-of-freedom (DOF), with $4$ degrees of actuation from the thrust and moment $(f,M)$. Thus the degrees-of-underactuation (DOuA) for the system is $(2n+2)$.
\end{remark}


\begin{remark}
The assumption that the flexible cable is a series of connected links may not be valid under some extreme conditions.
However, this assumption offers more flexibility over the assumption of a single mass-less link and can be potentially used to design more aggressive trajectories that require cable deformation.
\end{remark}

\subsection{Point-mass load suspended from multiple quadrotors through flexible cables (Fig.~\ref{fig:quad3})}
The second system is a point mass load suspended by $p$ quadrotors through flexible cables as shown in Fig.~\ref{fig:quad3} where $p>1$. The configuration variables are the load position $x_L \in \mathbb{R}^3$, attitude of of each link in the flexible cable $q_{ij} \in S^2$ (here, $q_{ij}$ is the attitude of the $j^{th}$ link in the flexible cable attached the $i^{th}$ quadrotor) and attitude of the quadrotors $R_i \in SO(3)$. 
Positions of different links $x_{ij}$ and quadrotors $x_{i0}$ can be obtained from the kinematic relations given below,
\begin{align}
x_{ij} =&\ x_{i(j-1)} + l_{ij}q_{ij},  \label{eq:mQpoint0b} \\
x_L =&\ x_{i(n_i-1)} + l_{in_i}q_{in_i}, \label{eq:mQpoint0a} 
\end{align}

where $j\in \{ 1,\hdots,(n_i-1)\}$, $i \in \{1,2,\hdots,p\}$, $n_i$ is the number of links in the $i^{th}$ flexible cable, $l_{ij}$ is the length of the $j^{th}$ link of the $i^{th}$ flexible cable (i.e., the flexible cable attached to the $i^{th}$ quadrotor) and $x_L$ is the load position.

Multiple quadrotors and multiple links in each cable result in a complicated system with high degree-of-underactuation. 
The configuration space of this system is given as $Q\coloneqq\mathbb{R}^3\times \Pi_{j=1}^p \left(SO(3) \times (S^2)^{n_j}\right)$.
The corresponding system dynamics can be described in terms of internal tensions $T_{ij}>0$ shown below, 
\begin{align}
\label{eq:mQpoint1}
m_i(\ddot{x}_{i0} + ge_3) =&\  f_iR_ie_3 + T_{i1}q_{i1},
\end{align}
\begin{align}
\label{eq:mQpoint2}
m_{ij}(\ddot{x}_{ij}+ge_3) =&\  -T_{ij}q_{ij} + T_{i(j+1)}q_{i(j+1)},
\end{align}
\begin{align}
\label{eq:mQpoint3}
m_L(\ddot{x}_L + ge_3 )=&\  -\sum_{i=1}^{n}{T_{in_i}q_{in_i}},
\end{align}
\begin{align}
\label{eq:mQpoint4}
J_i\dot{\Omega}_i + \hat{\Omega}_iJ_i\Omega_i =&\ M_i,
\end{align}
for $i \in \{1,\hdots,n\},\,j\in \{1,\hdots,(n_i-1)\}$ where $m_L$ is mass of the load and $m_i$, $J_i$, $f_i$ \& $R_i$ are the mass, intertia matrix, thrust and rotation matrix of the $i^{th}$ quadrotor.
\begin{remark}
This system has $\big(3+ 3p+2\sum_1^p n_i\big)$ DOF and $4p$ actuators. Thus, DouA in the system is $\big(3-p+2\sum_1^p n_i\big)$. For the configuration in Fig.~\ref{fig:quad3} with four quadrotors \& five link cables, DOF = 55 and DOuA = 39. 
\end{remark}
\subsection{Rigid-body load suspended from multiple quadrotors through flexible cables (Fig.~\ref{fig:quad4})}
\begin{figure}[t]
\centering
\includegraphics[width=\columnwidth]{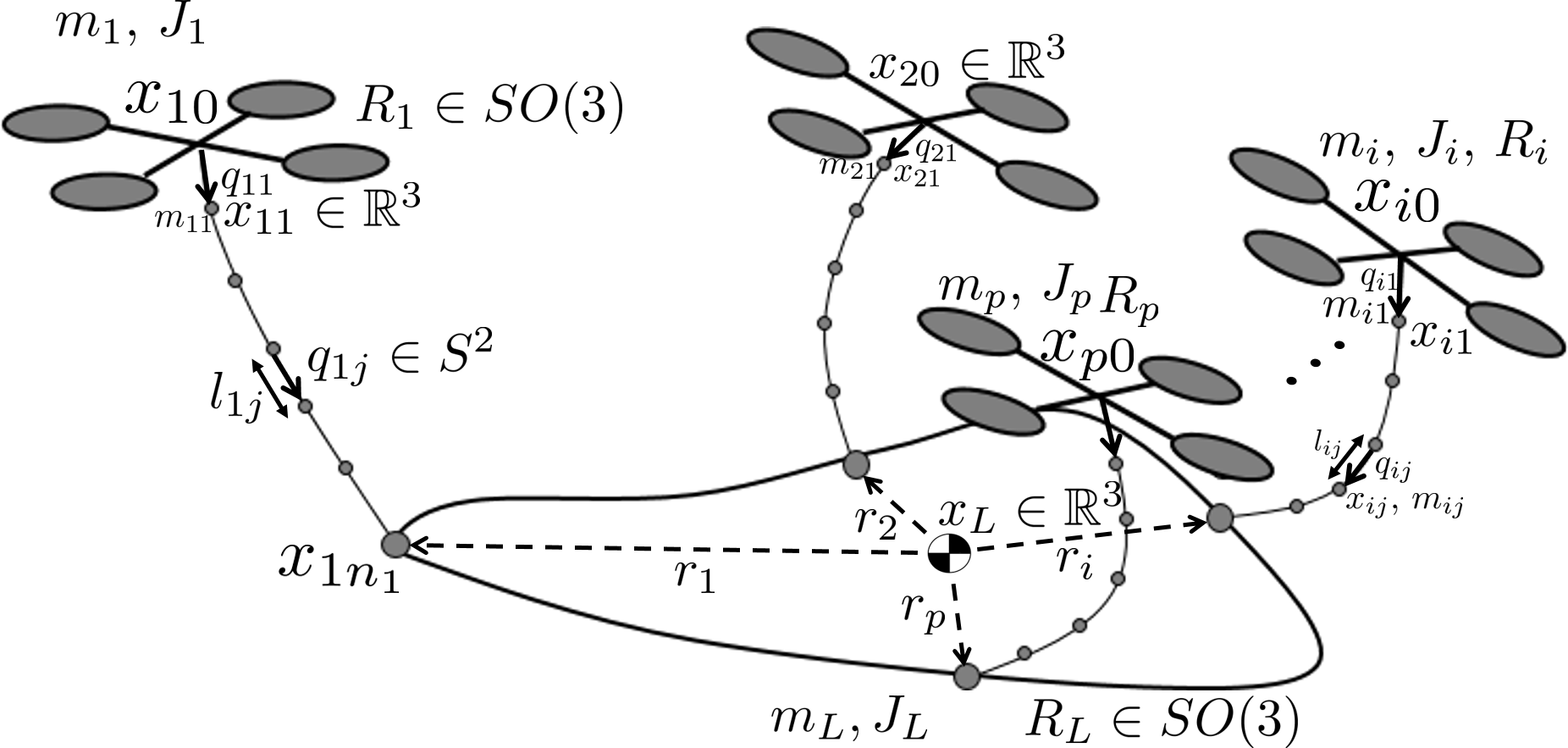}
\caption{Multiple quadrotors with a shared rigid-body payload supsended through flexible cables.}
\label{fig:quad4}
\end{figure}

We now consider the last system, where a rigid-body payload with mass $m_L$, inertia matrix $J_L$ and orientation $R_L$ in inertial frame, is suspended by $p$ quadrotors through flexible cables. Fig.~\ref{fig:quad4} shows the geometry of the load suspended from $p$ quadrotors by flexible cables. 
The kinematic relations between positions of cable links, quadrotors and the load position is given as follows,
\begin{align}
\label{eq:mQrigid0a}
x_{i(n_i-1)} =&\ x_L + R_Lr_i - l_{in_i}q_{in_i},
\end{align}
\begin{align}
\label{eq:mQrigid0b}
x_{i(j-1)} =&\ x_{ij} - l_{ij}q_{ij}, 
\end{align}
for $i \in \{1,\hdots,p\},\,j\in \{1,\hdots,(n_i-1)\}$ where $n_i$ is the number of links in the $i^{th}$ flexible cable, $x_L$ is the center-of-mass position of the load, $x_{i0}$ is the center-of-mass position of the $i^{th}$ quadrotor and $r_i$ is the vector (in the load body frame) from the load center of mass to the point of attachment of the last link of the $i^{th}$ flexible cable.

For the case of a rigid-body load, the degrees-of-freedom is increased by 3 due to the load attitude, compared to the case of point-mass load. 
Similar to the previous system, we can use the internal tensions $T_{ij}$ on the corresponding configuration space $Q\coloneqq{SE(3)} \times \Pi_{j=1}^p \left(SO(3) \times (S^2)^{n_j}\right)$ to express the dynamics shown below, 
\begin{align}
\label{eq:mQrigid1}
m_{i}\ddot{x}_{i0} =&\ f_iR_ie_3 - m_ige_3 + T_{i1}q_{i1},
\end{align}
\begin{align}
\label{eq:mQrigid2}
J_i\dot{\Omega}_i + \hat{\Omega}_iJ_i\Omega_i =&\ M_i,
\end{align}
\begin{align}
\label{eq:mQrigid3}
m_{ij}(\ddot{x}_{ij}+ge_3) =&\ -T_{ij}q_{ij} + T_{i(j+1)}q_{i(j+1)},
\end{align}
\begin{align}
\label{eq:mQrigid4}
m_L\ddot{x}_L =&\ -\sum_{i=1}^n{T_{in_i}q_{in_i}}-m_Lge_3,
\end{align}
\begin{align}
\label{eq:mQrigid5}
J_L\dot{\Omega}_L + \hat{\Omega}_LJ_L\Omega_L =&\ -\sum_{i=1}^n(r_i \times R_L^T T_{in_i}q_{in_i}),
\end{align}
for $i \in \{1,\hdots,p\},\,j\in \{1,\hdots,(n_i-1)\}$ and all other symbols with same representation as in the previous subsection. 

\begin{remark}
This system has $\big(6+3p+2\sum_1^p n_i\big)$ DOF and has $4p$ inputs and thus, $\big(6-p+2\sum_1^p n_i\big)$ DOuA. For the configuration shown in Fig.~\ref{fig:quad4}, with four quadrotors and 5 link cables, DOF = 58 and DOuA = 42.
\end{remark}

Having presented the dynamical models of the three systems considered in the paper, we next show that these systems are differentially-flat.



\section{DIFFERENTIAL FLATNESS}
\label{sec:diff-flat}

Differential flatness is the property of nonlinear systems, which identifies certain flat outputs for the system, such that all the system states and the inputs can be expressed as smooth functions of flat outputs and a finite number of their derivatives. 

Differentially flat systems have properties which can be used for feedback linearization. This concept has been previously exploited, to develop trajectories and achieve trajectory tracking control for loads suspended from quadrotors in \cite{mellinger2011minimum}, \cite{sreenath2013geometric} and \cite{sreenath2013trajectory}. 
According to \cite{murray1995differential}, differential flatness is defined as,
\begin{defn}
\label{defn:1}
{\it Differentially-Flat System \cite{murray1995differential}}: A system $\dot{\mathbf{x}} = f(\mf{x},\mf{u}),\,\mf{x} \in \mathbb{R}^n,\,\mf{u}\in \mathbb{R}^m,\,$ is differentially flat if there exists flat outputs $\mf{y}\in \mathbb{R}^m$ of the form $\mf{y} = \mf{y}(\mf{x},\mf{u},\dot{\mf{u}},\hdots,\mf{u}^{(\mf{p})})$ such that the states and the inputs can be expressed as $\mf{x} = \mf{x}(\mf{y},\dot{\mf{y}},\hdots,\mf{y}^{\mf{(q)}})$, $\mf{u} = \mf{u}(\mf{y},\dot{(\mf{y})},\hdots,\mf{y}^{(\mf{q})})$, where $\mf{p},\,\mf{q}$ are nonnegative integers.
\end{defn}
The following subsections show the differential flatness for different systems described in the previous section. Note that our work is different from \cite{murray1996trajectory} where the aerial agent dynamics are approximated as a fully-actuated point-mass.

\subsection{Quadrotor with load suspended through a flexible cable}

\begin{lemma}
\label{lemma:1} $\pazocal{Y} = (x_n, \psi)$ are the set of flat-outputs for the quadrotor with point mass load suspended through flexible cables, where $x_n \in \mathbb{R}^3$ is the position of the load (the $n^{th}$ point mass) and $\psi \in \mathbb{R}$ is the yaw angle of the quadrotor.
\end{lemma}
\begin{proof}
The tension vector in the $n^{th}$ link, $\vec{T}_n = T_nq_n$, can be calculated from \eqref{eq:diff3} since the $\ddot{x}_n$ is known from the flat-output $x_n$. The unit vector along the $n^{th}$ link and magnitude of the tension $T_n$ can be determined as $q_n = (T_nq_n)/\lVert T_nq_n \rVert$ and $T_n = (T_nq_n).q_n$. Tensions in all the remaining $(n-1)$ links can be calculated from \eqref{eq:diff2} iteratively. Positions of all other links and quadrotor position can be determined from \eqref{eq:dynA}. Since $(x_0, \psi)$ are the flat-outputs of a quadrotor \cite{mellinger2011minimum}, the rest of the states  $(R, \Omega)$ and inputs $(f, M)$ can be calculated. 
\end{proof}
\begin{remark}
To completely define and calculate all the states and inputs of the above system with $n-$ chain links, requires $(2n + 4)$ derivatives of the flat-output $x_n$ and $2^{nd}$ derivative of the yaw angle $\psi$.
\end{remark}
\begin{cor}
\label{cor}
$\pazocal{Y} = (x_n, \psi, \vec{F})$ are the flat-outputs for quadrotor with point mass load suspended via flexible cable with an external force $(\vec{F})$ acting on the point-mass load, where $x_n \in \mathbb{R}^3$ is the position of the load ($n^{th}$ point-mass) and $\psi \in \mathbb{R}$ is the yaw angle of the quadrotor.
\end{cor}
\begin{remark}
$\vec{F}$ in Corollary \ref{cor}, is an input to the system and is also a flat output. Note that from Definition \ref{defn:1}, the flat-output is a function of inputs, i.e, $\mf{y} = \mf{y}(\mf{x},\mf{u},\dot{\mf{u}},\hdots,\mf{u}^{(\mf{p})})$. Thus, for the system in Corollary \ref{cor}, the number of flat-outputs $=7$, ($x_n\in \mathbb{R}^3$, $\psi\in \mathbb{R}$, $\vec{F}\in \mathbb{R}^3$), and is equal to the number of inputs $=7$, ($M\in \mathbb{R}^3$, $f\in \mathbb{R}$, $\vec{F}\in \mathbb{R}^3$).
\end{remark}
%
%
%
\subsection{Point-mass load suspended from multiple quadrotors $(p \geq 1) $ through flexible cables with $(n_i\geq 1)$ links for $i\in\{1,2,...,n\}$.}
\begin{lemma}
\label{lemma:3}
${\bf \pazocal{Y}} = (x_L, T_{in_i}q_{in_i}, \psi_j)$ for $i \in \{2,\hdots,p\}$ and $j \in \{1,\hdots,p\}$ are the flat-outputs for the given system, with $\vec{T}_{in_i} = T_{in_i}q_{in_i} \in \mathbb{R}^3$ the tension in the last link of $(p-1)$ cables and $\psi_j \in \mathbb{R}$ the yaw angle of the quadrotors. $x_L \in \mathbb{R}^3$ is the load position.
\end{lemma}
\begin{proof}
From flat-output $x_L$ and its higher derivatives we can calculate $\sum{T_{in_i}q_{in_i}}$ from ~\eqref{eq:mQpoint3}. Knowing the values of $T_{in_i}q_{in_i}$ and its higer derivatives for $i\in \{2,\hdots,n\}$ we can calculate the value of $T_{1n_1}q_{1n_1}$ and its higher derivatives. Positions of different links of the systems can be calculated from \eqref{eq:mQpoint0b}-\eqref{eq:mQpoint0a}. Thus, we know the positions and tensions of the last link and their derivatives. 
The rest of the proof follows from Lemma \ref{lemma:1} and Corollary \ref{cor}.
\end{proof}
\begin{remark}
To completely describe all states and inputs of the above system as a function of flat-outputs requires, $(4+ 2n_{max})$, {\it(where, $n_{max} = \max\{n_1,\hdots,n_p\}$)}, derivate of $x_L$, $(2 + 2n_i)^{th}$ derivative of  $T_{in_i}q_{in_i}$ for $i\in \{2,\hdots,p \}$ and $2^{nd}$ derivative of $\psi_j$.
\end{remark}
\subsection{Rigid body load suspended from multiple quadrotors $(p \geq 1) $  through flexible cables with $(n_i \geq 1) $ links for $i\in \{1,\hdots,n\}$}

\begin{lemma}
\label{lemma:4}
$\pazocal{Y} = (x_L, R_L, \Lambda, \psi_j)$ for $j \in \{1,\hdots,p\}$ $(p\geq 3)$ is a set of flat outputs for the given system, where $\Lambda \in \mathbb{R}^{3p-6}$ satisfies, 
\begin{equation}
\label{eq:gensol}
\mathbb{T} = \Phi^\dagger W + N\Lambda
\end{equation}
with $\mathbb{T}, W$ defined as
\begin{equation}
\mathbb{T} = \begin{bmatrix}
R_L^TT_{1n_1}q_{1n_1} \\ R_L^TT_{2n_2}q_{2n_2} \\ \vdots \\ R_L^TT_{pn_p}q_{pn_p}
\end{bmatrix}_{3p\times 1}, W = -\begin{bmatrix}
R_L^Tm_L(\ddot{x}_L + ge_3) \\ J_L\dot{\Omega}_L + \hat{\Omega}_LJ_L\Omega_L
\end{bmatrix}_{6\times 1}
\end{equation}
and $\Phi^\dagger, N$ are respectively the Moorse-Penrose generalized inverse and the nullspace of 
\begin{equation}
\Phi = \begin{bmatrix}
I & I & \hdots & I \\ \hat{r}_1 & \hat{r}_2 & \hdots & \hat{r}_n 
\end{bmatrix}_{6\times 3p}
\end{equation}
provided that both $\Phi^\dagger_{3p\times 6}$ and $N_{3p\times (3p-6)}$ exist. 
\end{lemma}

\begin{proof}From \eqref{eq:mQrigid4} and \eqref{eq:mQrigid5}, we get,
\begin{equation}
\label{eq:tensions}
 -\begin{bmatrix}
R_L^Tm_L(\ddot{x}_L + ge_3) \\ J_L\dot{\Omega}_L + \hat{\Omega}_LJ_L\Omega_L
\end{bmatrix} = \Phi\begin{bmatrix}
R_L^TT_{1n_1}q_{1n_1} \\ R_L^TT_{2n_2}q_{2n_2} \\ \vdots \\ R_L^TT_{nn_n}q_{nn_n}
\end{bmatrix}.
\end{equation}
Proof follows from [Lemma 2, \cite{sreenathRSS2013}], where the tensions for the last links $(T_{in_i}q_{in_i})$ of each flexible cable are calculated. Note that the general solution to \eqref{eq:tensions}, is \eqref{eq:gensol}.
To compute the tensions and their higher order derivatives, we need the time-invariant matrices $\Phi^{\dagger}$ and $N$. Here, $\Phi^{\dagger} = (\Phi^T\Phi)^{-1}\Phi^T$ and $N$ is matrix whose columns span the kernel of $\Phi$, representing the constraints on the internal forces in the system.
Positions for links of cable can be calculated from ~\eqref{eq:mQrigid0a} and ~\eqref{eq:mQrigid0b}. Knowing position and tensions in the last link for each cable, from Lemma \ref{lemma:1}, rest of the states can be calculated. 
\end{proof}

\begin{remark}
Calculation of all the states and inputs for the system requires upto $2^{nd}$ derivative of $\psi_j$and $(2+2n_{max})$ {\it (where, $n_{max} = \max\{n_1,\hdots,n_p\}$)} derivates of $\mathbb{T}$,  which in turn depends on $W\,\&\,\Lambda$. Thus, we require $(2+2n_{max})$ derivatives of $\Lambda$ and $(4+2n_{max})$ derivatives of $x_L\, \&\, R_L$.
\end{remark}

So far we have discussed about differential flatness in different quadrotor-load with flexible cable systems. 
In the next section, we linearize the dynamics of load suspended from quadrotor using flexible cable about a specific desired time-varying trajectory. This linearization is performed directly on the manifolds and thus is singularity free.

%
%
%
%
%
%
%
%
%

\section{Control Design of a Single Quadrotor with Point-mass Load Suspended through A Flexible Cable}
\label{sec:control}
We have previously shown the equations of motion for quadrotor with a load supended through a flexible cable in \eqref{eq:diff1}-\eqref{eq:diff4}.
Designing a controller based on this model is not feasible since the values of each tension vector remain unknown.
Thus for the purpose of control, we use instead the compact geometric equations of motion developed in \cite{goodarzi2014geometric}. The system dynamics are linearized about a time-varying reference trajectory to obtain a linear time-varying dynamical model. However, since the system evolves on a complex manifold with the configuration space $Q\coloneqq\mathbb{R}^3\times SO(3)\times (S^2)^n$, standard linearization techniques is cumbersome to implement and involves complex calculations using local variables resulting in singularites. Variation based geometric linearization is used to overcome these difficulties, as discussed in \cite{Access2015_VariationLinearization}.
\subsection{Linearized Dynamics}
The equations of motion given in \eqref{eq:diff3}$-$\eqref{eq:diff4} is converted to a new compact representation as developed in \cite{goodarzi2014geometric} and  given below,
\begin{gather}
\begin{bmatrix}
M_{00}& M_{01} & M_{02}&\hdots &M_{0n}\\
-\hat{q}^2_1M_{10}& M_{11}I_3& -M_{12}\hat{q}^2_1&\hdots &-M_{1n}\hat{q}^2_1\\
-\hat{q}^2_2M_{20}& -M_{21}\hat{q}^2_2& M_{22}I_3 &\hdots &-M_{2n}\hat{q}^2_2\\
\hdots & \hdots & \hdots && \hdots\\
-\hat{q}^2_nM_{n0}& -M_{n1}\hat{q}^2_n& -M_{n2}\hat{q}^2_n& \hdots & M_{nn}I_3
\end{bmatrix}\begin{bmatrix}
\ddot{x}_0 \\ \ddot{q}_1 \\ \ddot{q}_2 \\ \hdots \\ \ddot{q}_n
\end{bmatrix}\nonumber \\
=\begin{bmatrix}
fRe_3 - M_{00}ge_3 \\ 
-\|\dot{q}_1\|^2M_{11}q_1 + \sum^n_{a=1}m_agl_1\hat{q}^2_1e_3 \\
-\|\dot{q}_2\|^2M_{22}q_2 + \sum^n_{a=2}m_agl_1\hat{q}^2_2e_3 \\
\hdots \\
-\|\dot{q}_n\|^2M_{nn}q_n + m_ngl_1\hat{q}^2_ne_3
\end{bmatrix}, \label{eq:tlee1} \\
\dot{q}_i = \omega_i \times q_i.\label{eq:tlee2}
\end{gather}
Using this compact representation, we linearize the dynamics about a desired trajectory using the variation techniques discussed in \cite{Access2015_VariationLinearization}. 
We list all the error states of the system as $s=\{ \eta,\, \delta \Omega,\, \delta x_0,\, \xi_1,\, \hdots ,\, \xi_n ,\,\delta v_0,\, \delta \omega_1,\,\hdots ,\,\delta \omega_n \}$, where $(\xi_i,\,\delta \omega_i)$ correspond to the linear error state  approximation for the direction vector of the $i^{th}$ link. Similarly, $\eta$ and $\delta\Omega$ are the linear error states approximations for attitude and body-angular velocity of the quadrotor. Detailed discussion about the error states and variations is presented in Appendix A. 
The linearized dynamics are given in the following equations: 
\begin{align}
\label{eq:linearEq}
\dot{s} =&\ A(t)s + B(t)\delta u,
\end{align}
\begin{align}
\label{eq:linearEqConst}
C(t)s =&\ 0,
\end{align}
where the state and the input are,
\begin{align}
\label{eq:linearEq1}
\begin{split}
s =&\ \begin{bmatrix}
\eta \, \delta \Omega \,\,  \delta x_0\,\,  \xi_1\,\,  \hdots \,\,  \xi_n \,\,  \delta v_0\,\, \delta \omega_1\,\,\hdots \,\,\delta \omega_n 
\end{bmatrix}^T\in \mathbb{R}^{12+6n},
\end{split}
\end{align}
\begin{align}
\label{eq:linearEq2}
\delta u =&\ \begin{bmatrix}
\delta f & \delta M
\end{bmatrix}^T\in \mathbb{R}^{m} = \mathbb{R}^4,
\end{align}
where the expressions of $A(t), B(t), C(t)$ are given in Appendix B and their derivation in Appendix C. $C(t)$ reflects the state constraints introduced due to the geometric structure of the manifold. More on this can be found in \cite{Access2015_VariationLinearization}.  
\begin{remark}
For a given load trajectory as a function of time, the desired states and feed forwards inputs $u_d = [f_d, M_d]^T$ can be calculated using the flat-ouputs as discussed in Lemma~\ref{lemma:1}.
\end{remark}
\subsection{Finite Horizon Linear Quadratic Regulator (LQR)}
Note that the resulting linearized dynamics \eqref{eq:linearEq} is essentially a time-varying linear system, since it is derived through variation based linearization about a desired trajectory which can be time-varying. ($A(t),B(t),C(t)$ from \eqref{eq:linearEq}, \eqref{eq:linearEqConst} are expressed in terms of the desired states $x_d, R_d, q_{id}$). Any standard control technique used for a linear system are applicable. Since the system is time-varying, we implement a finite-horizon LQR controller. 

The state $s(t)$ gives the linear error in the system, which can be calculated using \cite[Eq.~(2),(4)]{goodarzi2014geometric}

A finite-horizon $T$ is chosen along with the positive semi-definite matrices $Q_1 = Q_1^T \geq 0 \in \mathbb{R}^{12+6n\times 12+6n}$ and $Q_2 = Q_2^T\geq 0\in \mathbb{R}^{{m}\times{m}}$, where $Q_1$ and $Q_2$ are weight matrices corresponding to the states $s(t)$ and control inputs $\delta u$. We also choose the final weight matrix at $t=T$, $P(T) = P_T = P^{\it T}_T\geq 0 \in \mathbb{R}^{12+6n\times 12+6n}$ as the weight matrix for the terminal state $s(T)$. Where $n$ is the number of links in the cable. 

To solve for the optimal solution of the finite-horizon LQR, we need to first solve the continous-time Riccati equation given below,
\begin{align}
-\dot{P}(t) =&\ Q_1 - P(t)B(t)Q_2^{-1}B(t)^TP(t) 
\nonumber
\\
+&\ A(t)^TP(t) + P(t)A(t) 
\label{eq:lqr}.
\end{align}
\indent
For real-time implementation, we need to intergrate ~\eqref{eq:lqr} backwards in time from $t=T$ to $t=0$, with the terminal condition $P(T) = P_T$.
The precomputed values of $P(t)$ are stored in a table for calculating the feedback gain online. 
Then the value $P(t)$ is used to calculate feedback control input for the linear system \eqref{eq:linearEq} as,
\begin{equation}
\delta u(t) = -K(t)s(t) = -Q_2^{-1}B(t)^TP(t)s(t),
\end{equation}
\indent 
Finally, the trajectory tracking controller	can be calculated as,
\begin{align}
u(t) = u_d(t) + \delta u(t),
\end{align}
where $u(t) = [f(t), M(t)]^T$, and the gain matrix $K(t)$ can be computed online based on the stored values of $P(t), B(t)$ and $Q_2$. 

\begin{remark}
Finite-Horizon LQR is implemented on the variation-based linearized dynamics of the errors on the manifold. The controller is globally stable for the linearized dynamics \eqref{eq:linearEq} and is locally stable for the complete non-linear dynamics \eqref{eq:tlee1}-\eqref{eq:tlee2}.
\end{remark}

\section{SIMULATION RESULTS}
\label{sec:simulation}
\begin{figure*}[t!]
\centering
\begin{subfigure}{.65\columnwidth}
\includegraphics[width=0.75\columnwidth]{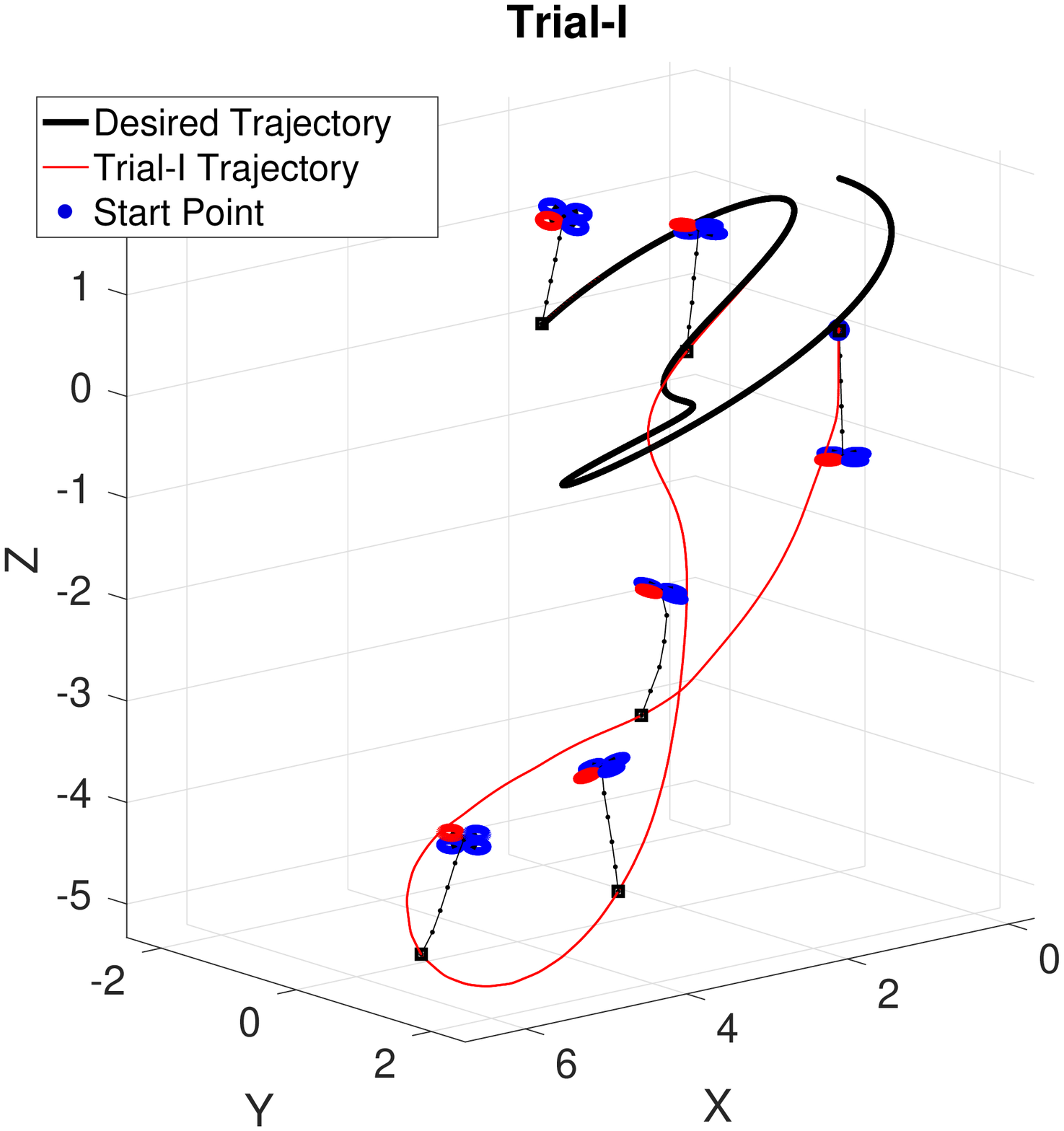}%
\caption{Trial-I}%
\label{fig:2a}%
\end{subfigure}\hfill%
\begin{subfigure}{.65\columnwidth}
\includegraphics[width=0.75\columnwidth]{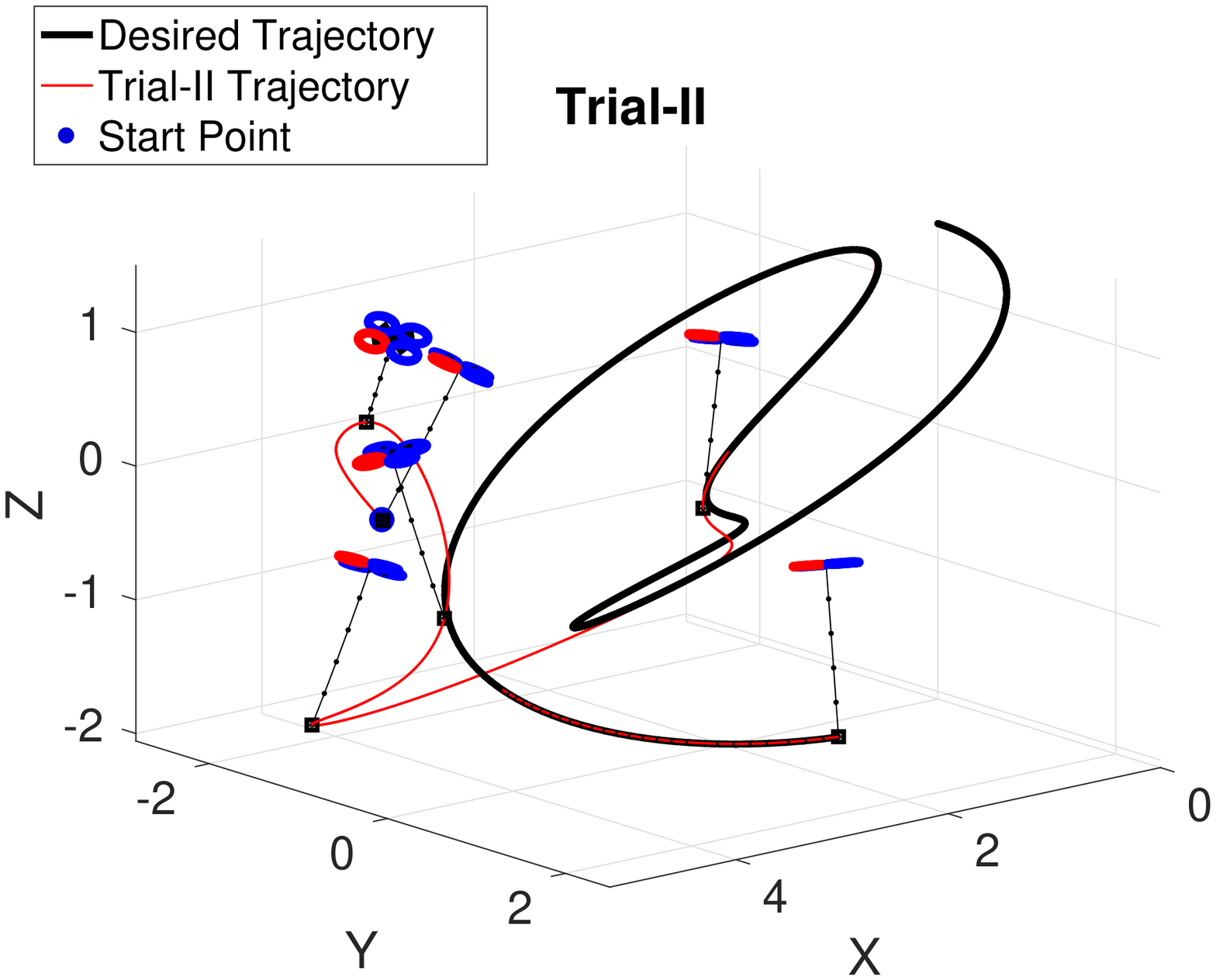}%
\caption{Trial-II}%
\label{fig:2b}%
\end{subfigure}\hfill%
\begin{subfigure}{.65\columnwidth}
\includegraphics[width=0.75\columnwidth]{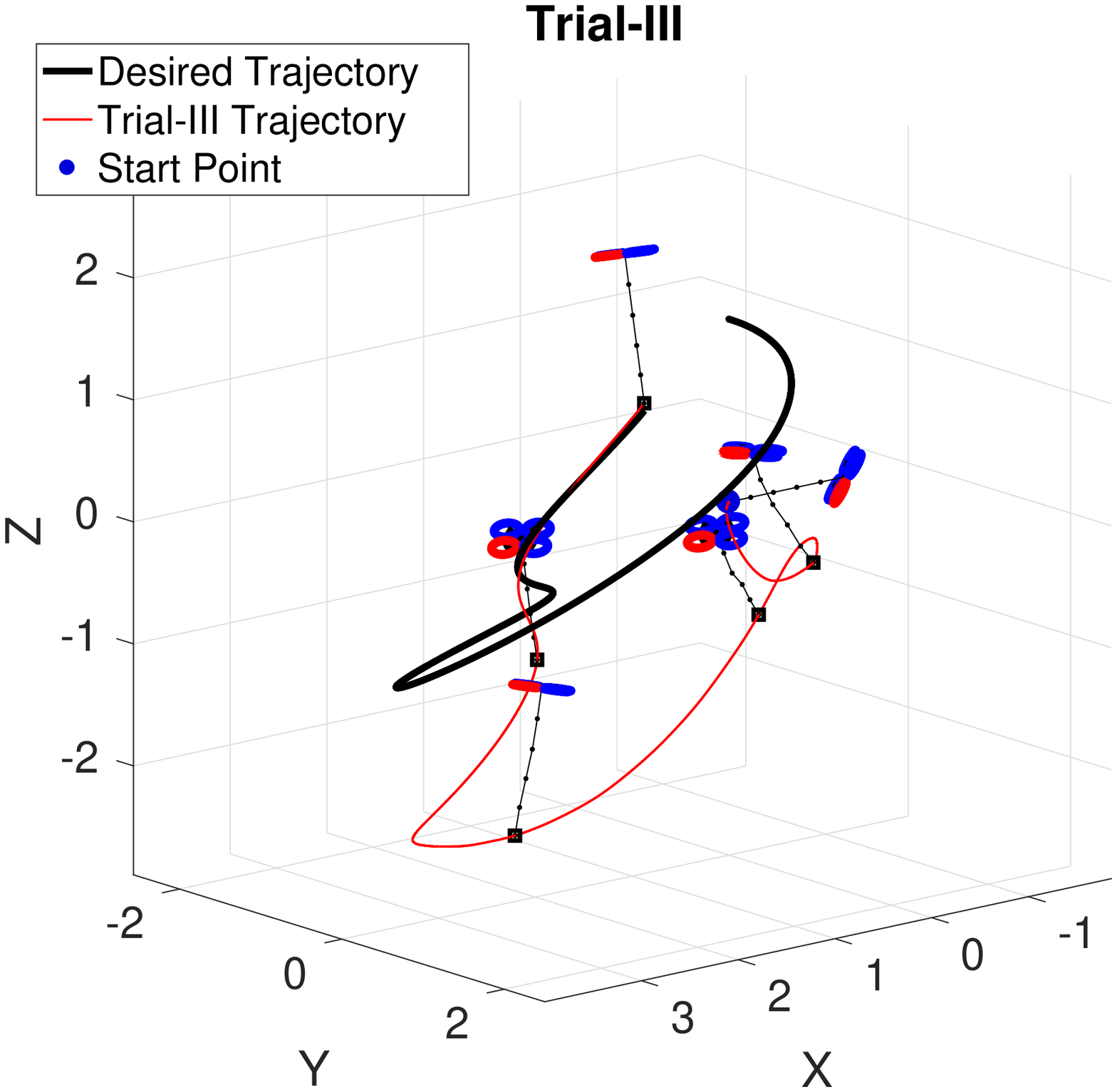}%
\caption{Trial-III}%
\label{fig:2c}%
\end{subfigure}%
\caption{Snapshots of quadrotor with load suspended by flexible cable at various instances of time along the trajectory (red) obtained through variation based linearization controller to track the reference trajectory (black). }
\label{fig:2}
\end{figure*}


\begin{figure}[t!]
\centering
\includegraphics[width=\columnwidth]{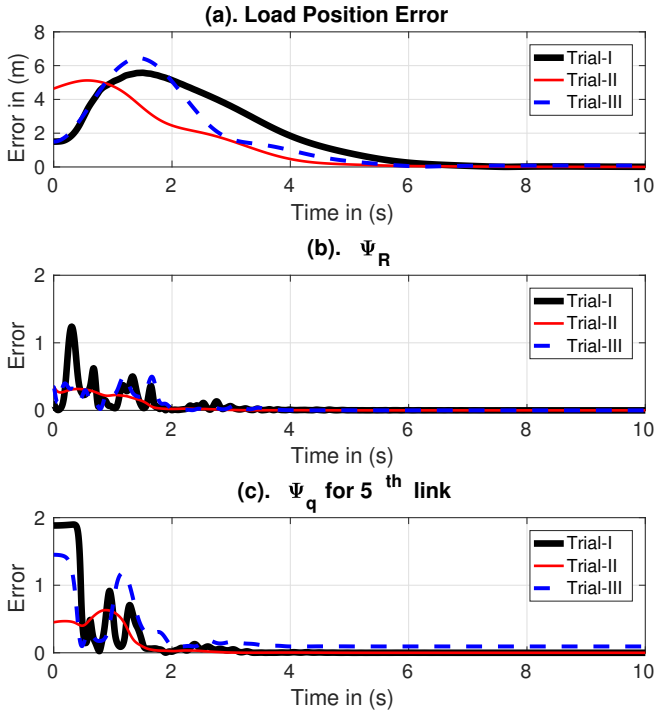}
\caption{Tracking errors obtained through simulation while tracking the reference trajectory using a LQR control designed based on the geometric variation-based linearized dynamics.}
\label{fig:3}
\end{figure}

Having presented the linearized dynamics and controllers, we now proceed to numerically validate it. 
In particular, we use the Matlab ode solver with 4th-order Runge-Kutta method and implement our controller. 
To study the performance of the controller developed in section \ref{sec:control}, we choose a moderately aggressive trajectory defined by flat outputs as,
\begin{align*}
x_n(t) = \begin{bmatrix}
a_x(1-\cos(2f_1\pi t)\\a_y\sin(2f_2\pi t) \\ a_z\cos(2f_3\pi t)
\end{bmatrix}, \quad 
\psi(t) \equiv 0,
\end{align*}
where {\small $a_x = 2,\,a_y = 2.5,\,a_z = 1.5,\,f_1 = \frac{1}{4},\,f_2 = \frac{1}{5},\,f_3=\frac{1}{7}$.}

The rest of the states and the nominal feedforward inputs required to track the trajectory are calculated through differential flatness. These states refer to the desired states used in calculating the errors and the values of $A,\ B$ in \eqref{eq:linearEq}. 
The controller performance is tested through several simulation tests.
In the simulation environment, the parameters of the quadrotor are given by
$
m_Q = 0.85kg,\quad J = diag([0.557, 0.557, 1.05])\times 10^{-2}kg.m^2
$
and the parameters of the cable are given as
$
n = 5,\quad m_i = 0.1\,kg,\quad l_i = 0.25m
$
for $i = 1,2,...,n$.
For the LQR controller, the weight matrices are chosen as,
\begin{equation*}
Q_1 = diag([Q_{11}, Q_{12}, Q_{13}, Q_{14}]),
\end{equation*}
where 
$
Q_{11} = 0.5I_6,\, Q_{12} = 0.75I_6,\, Q_{13} = I_{3n},\, Q_{14} = 0.75I_{3n}
$ and 
$
Q_2 = 0.2I_4,P_T = 0.01 \cdot I_{12+6n}
$.


Fig.~\ref{fig:2} illustrates the trajectories of the control system for three different initial conditions.
From Fig.~\ref{fig:2}, it can be seen that the trajectories for all three different initial conditions converge to the reference trajectory, even for the initial condition with large initial deviation.
This implies that the controller is still able to stabilize the trajectory to the reference, emphasizing that the linear controller developed through variation on manifolds has a large domain of attraction.  

Next, we loot at various position and attitude errors in Fig~\ref{fig:3}. In particular, the load position tracking error $\|(\delta x_n)\|_2$, Fig.~\ref{fig:3}a, rotation error $\Psi_R$ for quadrotor orientation, Fig.~\ref{fig:3}b and the orientation error $\Psi_q$ for $5^{th}$ link, Fig.~\ref{fig:3}c, converge to zero. (Definitions for $\Psi_R$ and $\Psi_{q_i}$ are given in \eqref{eq:so35}, \eqref{eq:s23} respectively.)
This validates that the controller developed for trajectory tracking of load suspended from a quadrotor by a flexible cable, through variation-based linearization. 

\begin{remark}
Discrete representation of the cable as discussed in the earlier sections captures the dynamics of the cable, however this increases the number of states in the system. These increased states makes the experimental implementation of the control harder to achieve, since the experimental implementation requires the measurement of the orientation of each link. We are working towards experimental validation as part of future work.
\end{remark}

\begin{figure*}[t]
\normalsize
\begin{equation}
\small
\label{eq:linearEqforA}
A = \begin{bmatrix}
\quad\quad\begin{bmatrix} 
-\widehat{\Omega}_d & I & O & O & \hdots & O  & O & O & \hdots & O \\
O & \Delta_1 & O & O & \hdots & O  & O & O & \hdots & O \\
O & O & O & O & \hdots & O & I & O & \hdots & O \\ 
O & O & O & \alpha_1 & \hdots & O & O & \beta_1 & \hdots & O \\
O & O & O & \vdots & \ddots & \vdots & \vdots & \vdots & \ddots & \vdots \\
O & O & O & O & \hdots & \alpha_n & O& O & \hdots & \beta_n \\
\end{bmatrix}\\
N^{-1}\begin{bmatrix}
\Delta_2 & O&O &a_{1} & \hdots & a_{n} &O& b_{1}&\hdots & b_{n} \\
O&O&O&c_{11} & \hdots & c_{1n} &O& d_{11}&\hdots & d_{1n}\\
O&O&O&c_{21} & \hdots & c_{2n} &O&  d_{21} & \hdots & d_{2n}\\
\vdots & \vdots &\vdots &\vdots  && \vdots &\vdots &\vdots & & \vdots\\
O& O&O &c_{n1}& \hdots & c_{nn} &O&  d_{n1} & \hdots & d_{nn}
\end{bmatrix}
\end{bmatrix},\quad
B = \begin{bmatrix}
\begin{bmatrix}
O_{3\times 1} & O\\
O_{3\times 1} & J^{-1}\\
O_{3\times 1} & O \\ 
O_{3n\times 1} & O_{3n\times 3} 
\end{bmatrix}\\ 
N^{-1}\begin{bmatrix}
R_de_3& O_{3\times 3} \\ O_{3n\times 1}& O_{3n\times 3 }
\end{bmatrix}
\end{bmatrix},
\end{equation}
\vspace*{4pt}
\end{figure*}

\section{CONCLUSIONS}
\label{sec:conclusion}
We have studied the payload transportation problem of multiple quadrotors with the payload suspended through flexible cable(s).  In particular, we have considered the following systems:
(a) single quadrotor with a point-mass payload suspended through a flexible cable; 
(b) multiple quadrotors with a shared point-mass payload suspended through flexible cables; and (c) multiple quadrotors with a shared rigid-body payload suspended through flexible cables.
For each of these systems, we have developed the Newton-Euler coordinate-free dynamic models and proven that the resulting dynamics are differentially-flat.
For the single quadrotor with a point-mass payload suspended through a flexible cable with five links (16 degrees-of-freedom and 12 degrees-of-underactuation), we have used the coordinate-free dynamics to develop a geometric variation-based linearized equations of motion about a desired trajectory. We show that a finite-horizon linear quadratic regulator, designed based on the linearized dynamics, can be used to track a desired trajectory with a relatively large region of attraction.  We demonstrate this through several numerical simulations. Control design for the rest of the systems will be presented in future work.

%

\section*{APPENDIX}
\subsection{Variation Expressions}
\label{sec:math}
The distance between points on a manifold can be measured through the concept of configuration error.
The infinitesimal variations can be considered as a linear approximation of this configuration error on the manifold. 
Geometric linearization is to get the dynamics of the infinitesimal variations in the form of a linear system.
For the purpose of control, we could roughly treat the variation as the error between the planned trajectory and the actual state. 
The corresponding expressions on $\mathbb R^3$, $S^2$ and $SO(3)$ are given as follows. 

\begin{remark}
The subscript $\bm{d}$ in the rest of the section refers to the time-varying desired reference trajectory. For a given sufficiently smooth load trajectory profile, desired states and feed-forward inputs can be easily calculated for differentially flat systems.
\end{remark}

\subsubsection{Variation in $\mathbb{R}^3$}
Infinitesimal variation in Cartesian space $\mathbb{R}^3$ with respect to a reference position vector $x_d(t)\in \mathbb{R}^3$ and velocity $v_d(t)\in \mathbb{R}^3$ are,
\begin{equation*}
\delta x(t) = x(t)-x_d(t),\quad \delta v(t) = v(t) - v_d(t).
\end{equation*}
\indent
For such flat space, the linear error state in $\mathbb{R}^3$ is the exact distance as, 
\begin{equation}
\label{eq:var3}
s = \begin{bmatrix}
\delta x \\ \delta v
\end{bmatrix} = \begin{bmatrix}
x(t)-x_d(t) \\ v(t) - v_d(t).
\end{bmatrix}
\end{equation}

\subsubsection{Variation in $S^2$}
Infinitesimal variation in $S^2$ with respect to a desired unit direction $q_{id}(t)\in S^2$ can be calculated as, 
\begin{equation}
\label{eq:s21}
\delta q_i(t) = \frac{d}{d\epsilon}\Big| _{\epsilon=0}e^{(\epsilon\widehat{\xi}_i)}q_{id}(t) = {\xi}_i\times q_{id}(t),
\end{equation}
where $\xi_i\in \mathbb{R}^3$, subject to $\xi_i.q_{id} = 0$ and
\begin{equation}
\delta\omega_i\cdot q_{id} + \omega_{id}\cdot (\xi_i\times q_{id}) = 0.
\end{equation}
%
\indent 
If the actual direction $q_i(t)$ is close to the desired direction vector $q_{id}(t)$, we can approximate the linear error states $[\xi_i, \delta \omega_i]$ to the errors, $e_{q_i}$, $e_{\omega_i}$ on $S^2$ (see \cite{Access2015_VariationLinearization}).
\begin{equation}
\label{eq:var2}
s = \begin{bmatrix}
\xi_i \\ \delta \omega_i
\end{bmatrix} \approx \begin{bmatrix}
e_{q_i} \\ e_{\omega_i}
\end{bmatrix} = \begin{bmatrix}
\widehat{q}_{id}(t)q_i(t) \\ \omega_i(t)-(-\hat{q}_i^2(t)\omega_{id}(t))
\end{bmatrix}.
\end{equation}
The configuration error for the cable link's direction on $S^2$ is given as,
\begin{align}
\label{eq:s23}
\Psi_{q_i} = (1-q_i.q_{id}).
\end{align}

\subsubsection{Variation in $SO(3)$}
Infinitesimal variation in $SO(3)$ with respect to a desired rotation matrix $R_d(t)\in SO(3)$ can be calculated as,
\begin{equation}
\label{eq:so31}
\delta R(t) = \frac{d}{d\epsilon}\Big| _{\epsilon=0}R_de^{(\epsilon\widehat{\eta})} = R_d(t)\widehat{\eta},
\end{equation}
where $\eta \in \mathbb{R}^3$. In a similar manner, the infinitesimal variation of body-angular velocities is given as,
\begin{equation}
\label{eq:so32}
\delta \Omega(t) = \widehat{\Omega}_d(t)\eta(t) + \dot{\eta}(t).
\end{equation}
%
%
%

If the actual rotation matrix $R(t)$ is close to the desired rotation matrix $R_d(t)$, it can be assumed that $[\eta,\,\delta\Omega]$ are linear approximation of the error $\begin{bmatrix}e_R,e_{\Omega} \end{bmatrix}^T$ between the actual and desired rotation matrices and angular velocities  (see \cite{Access2015_VariationLinearization}). 
Then we denote the error state as, 
\begin{equation}
\label{eq:var1}
s = \begin{bmatrix}
\eta \\ \delta\Omega
\end{bmatrix} \approx
\begin{bmatrix}
e_R \\ e_{\Omega}
\end{bmatrix} = \begin{bmatrix}
\frac{1}{2}\big(R_d^T(t)R(t) - R^T(t)R_d(t)\big)^\vee \\ \Omega(t) - \big(R^T(t)R_d(t)\big)\Omega_d(t)
\end{bmatrix}.
\end{equation}

The configuration error for the quadrotor rotation matrix on $SO(3)$ is given below,
\begin{equation}
\label{eq:so35}
\Psi_R = \frac{1}{2}(trace(I - R_d^TR)).
\end{equation}

\subsection{Linearized Dynamics}
In this subsection, we present the expressions for A, B  and C of the linear system in \eqref{eq:linearEq} with the derivation given in Appendix C. Expression for A and B are given in \eqref{eq:linearEqforA}. Here, 

$\Delta_1 =  J^{-1}(\widehat{J\Omega_d}-\widehat{\Omega}_dJ)$, $\Delta_2 = -f_dR_d\hat{e}_3$\\
and $\alpha_i = q_{id}q_{id}^T\widehat{\omega}_{id}$, $\beta_i = (I - q_{id}q_{id}^T),$
with,
{\small
\begin{equation*}
a_i = M_{0i}(\widehat{\dot{\omega}}_{id}- \lVert\omega_{id}\rVert^2I)\widehat{q}_{id},
\end{equation*}
\begin{equation*}
b_i = M_{0i}(2q_{id}\omega^T_{id}),
\end{equation*}
\[ c_{ij} =
  \begin{cases}
    \big[ M_{i0}\widehat{\ddot{x}}_{0d}-\sum^n_{k=1,k\neq i}M_{ik}(\widehat{\widehat{q}_{kd}\dot{\omega}}_{kd} +\\ \quad\quad \lVert\omega_{kd}\rVert^2\widehat{q}_{kd}) +\sum^n_{a=i}m_agl_i\hat{e}_3 \big](-\widehat{q}_{id}),       & \quad i = j\\
     M_{ij}\widehat{q}_{id}(\widehat{\dot{\omega}}_{jd}-\lVert\omega_{jd} \rVert^2I)\widehat{q}_{jd}, \quad & \quad i \neq j\\
  \end{cases}
\]
\[ d_{ij} =
  \begin{cases}
    O_{3 \times 3},\quad & \quad  i = j\\
    2M_{ij}\hat{q}_{id}q_{jd}\omega^T_{jd},\quad   & \quad   i \neq j\\
  \end{cases}
\]}
and 
{\footnotesize
\begin{multline}
\label{eq:eomD}
N = \\
\begin{bmatrix}
M_{00}I & -M_{01}\hat{q}_{1d} & -M_{02}\hat{q}_{2d} & \hdots & -M_{0n}\hat{q}_{nd}\\
\hat{q}_{1d}M_{10} & M_{11}I & -M_{12}\hat{q}_{1d}\hat{q}_{2d}  & \hdots & -M_{1n}\hat{q}_{1d}\hat{q}_{nd}\\
\hat{q}_{2d}M_{20} & -M_{21}\hat{q}_{2d}\hat{q}_{1d}  & M_{22}I & \hdots & -M_{2n}\hat{q}_{2d}\hat{q}_{nd}\\
\vdots & \vdots & \vdots & & \vdots\\
\hat{q}_{nd}M_{n0} & -M_{n1}\hat{q}_{nd}\hat{q}_{1d}  & -M_{n2}\hat{q}_{nd}\hat{q}_{2d}  & \hdots & M_{nn}I
\end{bmatrix}.
\end{multline}
}
The constraint matrix, {\footnotesize
\begin{equation}
C = \begin{bmatrix}
O_{n\times 9 } & diag([C1_1,\hdots,C1_n]) & O_{n\times 3n}\\
O_{n\times 9 } & diag([C2_1,\hdots,C2_n]) &diag([C3_1,\hdots,C3_n])\\
\end{bmatrix}
\end{equation}}

with,
\begin{equation*}
C1_i = q^T_{id},\quad C2_i = -\omega^T_{id}\widehat{q}_{id},\quad C3_i = q^T_{id}.
\end{equation*}
Here, $diag([\,])$ represents a block diagonal matrix. $O_{i\times j}$ refers to the zero matrix with size $i\times j$, where as $O$ is a zero matrix of size $3\times 3 $. Finally $I$ refers to the Identity matrix of size $3\times 3$.

\subsection{Derivation for linearized system dynamics}
\label{appendix:B}
Linearized equations of motion is provided in \eqref{eq:linearEq} and in Appendix B. Here we present the detailed derivation for the variation based linearization for quadrotor with load suspended by flexible cable. Equation \eqref{eq:tlee1} can be separated into \eqref{eq:b1} \& \eqref{eq:b2},{\small
\begin{gather}
M_{00}I\dot{v}_0 -\sum^n_{j=1}M_{0j}\widehat{q}_j\dot{\omega}_j
= \sum^n_{j=1}M_{0j}\lVert \omega_j \rVert^2q_j+fRe_3 - M_{00}ge_3, \label{eq:b1}
\end{gather}}
and
{\small
\begin{gather}
\widehat{q}_iM_{i0}\dot{v}_0 - \sum^n_{j=1,j\neq i}M_{ij} \widehat{q}_i\widehat{q}_j\dot{\omega}_j + M_{ii}I\dot{\omega}_i \nonumber \\
= \sum^n_{j=1,j\neq i}M_{ij}\lVert\omega_j\rVert^2\widehat{q}_iq_j - \sum^n_{a=i}m_agl_i\widehat{q}_ie_3, \label{eq:b2}
\end{gather}}
$\forall\, i=\{1,2,..,n \}$. Taking variation on \eqref{eq:b1} about a desired trajectory results in, {\small
\begin{gather}
M_{00}I(\delta\dot{v}_0) -\sum^n_{j=1}M_{0j}\big[(\delta\widehat{q}_j)\dot{\omega}_{jd} + \widehat{q}_{jd}(\delta\dot{\omega}_{j})\big] \nonumber \\
= \sum^n_{j=1}M_{0j}\big[\delta(\lVert \omega_j\rVert^2)q_{jd} + \lVert \omega_{jd}\rVert^2\delta(q_{jd})\big]\nonumber \\
+f_d(\delta R)e_3 + (\delta f)R_de_3 -\delta(M_{00}ge_3).
\label{eq:b1a}
\end{gather}
From \eqref{eq:so31} and \eqref{eq:s21}, we have $\delta R = R_d\widehat{\eta}$ and $\delta q_j = \widehat{\xi}_jq_{jd} \implies \delta q_j = - \widehat{q}_{jd}\xi_j$ and also $\lVert \omega_j \rVert^2 = \omega_j^T\omega_j \implies \delta(\lVert \omega_j \rVert^2) = 2{\omega}_j^T(\delta \omega_j)$. Substituting these in \eqref{eq:b1a} and simplifying we get,\\
{ \small
\begin{gather}
M_{00}I(\delta\dot{v}_0) -\sum^n_{j=1}M_{0j} \widehat{q}_{jd}(\delta\dot{\omega}_{j}) =  \sum^n_{j=1}M_{0j}\big[(\widehat{\dot{\omega}}_{jd}- \lVert\omega_{jd}\rVert^2)\widehat{q}_{jd}\big]\xi_j \nonumber \\
+\sum^n_{j=1}M_{0j}(2q_{jd}\omega^T_{jd})(\delta \omega_j) -f_dR_d\widehat{e}_3\eta+ (\delta f)R_de_3. \label{eq:b1b}
\end{gather}}
Similarly applying variation to \eqref{eq:b2}, we have,{\small
\begin{gather}
M_{i0}\big[(\delta\widehat{q}_i)\dot{v}_0+\widehat{q}_i(\delta\ddot{x}_0)\big] + M_{ii}I(\delta\dot{\omega}_i) \nonumber \\
-\sum^n_{j=1,j\neq i}M_{ij}\big[ (
\delta\widehat{q}_i)\widehat{q}_{jd}\dot{\omega}_{jd} + \widehat{q}_{id}(\delta\widehat{q}_j)\dot{\omega}_{jd}+ \widehat{q}_{id}\widehat{q}_{jd}(\delta \dot{\omega}_j)\big] \nonumber \\
= \sum^n_{j=1,j\neq i}M_{ij}\big[\lVert\omega_{jd}\rVert^2\widehat{q}_{id}(\delta q_j)+ \lVert\omega_{jd}\rVert^2(\delta\widehat{q}_i)q_{jd} \nonumber \\
+ 2\widehat{q}_{id}q_{jd}\omega^T_{jd}(\delta \omega_j)\big]+\sum^n_{a=i}m_agl_i\hat{e}_3(\delta q_i), \label{eq:b2a}
\end{gather}}
and simplifying it results,{\small
\begin{gather}
M_{i0}\widehat{q}_{id}(\delta\dot{v}_0) - \sum^n_{j=1,j\neq i}M_{ij} \widehat{q}_{id}\widehat{q}_{jd}(\delta\dot{\omega}_j) + M_{ii}I(\delta\dot{\omega}_i) \nonumber \\
= \bigg[ M_{i0}\widehat{\dot{v}}_{0d}-\sum^n_{k=1,k\neq i}M_{ik}(\widehat{\widehat{q}_{kd}\dot{\omega}}_{kd} + \lVert\omega_{kd}\rVert^2\widehat{q}_{kd})\nonumber \\
 +\sum^n_{a=i}m_agl_i\hat{e}_3 \bigg](-\widehat{q}_{id})\xi_i + \sum^n_{j=1,j\neq i}\bigg[M_{ij}\widehat{q}_{id}(\widehat{\dot{\omega}}_{jd}-\lVert\omega_{jd} \rVert^2)\widehat{q}_{jd} \bigg]\xi_j\nonumber \\ 
 + \sum^n_{j=1,i\neq j}\bigg[2M_{ij}\hat{q}_{id}q_{jd}\omega^T_{jd} \bigg](\delta \omega_j).
\label{eq:b2b}
\end{gather}}

Derivatives of the variations $\delta x_0$, $\delta q_i$, and $\delta R$ are as given below,
\begin{equation}
\label{eq:b3}
\delta v_0 = \delta \dot{x}_0,\quad \delta \dot{v}_0 = \delta \ddot{x}_0
\end{equation}
\begin{equation*}
\begin{split}
(\delta\dot{q}_i) = (\delta\omega_i) \times q_{id} + \omega_{id}\times(\delta q_i)\\
-\hat{q}_{id}\dot{\xi}_i - \hat{\dot{q}}_{id}\xi_{i} = -\hat{q}_{id}(\delta\omega_i) - \hat{\omega}_{id}\hat{q}_{id}\xi_i
\end{split}
\end{equation*}
\begin{equation}
\label{eq:b4}
\dot{\xi}_i = (q_{id}q^T_{id}\hat{\omega}_{id})\xi_i + (I - q_{id}q^T_{id})(\delta \omega_i)
\end{equation}
\begin{equation*}
(\delta\dot{R}) = (\delta R)\hat{\Omega}_d + R_d(\delta\Omega),
\end{equation*}
substituting values for $\delta R$ \& $\delta \dot{R}$ gives,
\begin{equation}
\label{eq:b5}
\dot{\eta} = -\hat{\Omega}_d\eta + I(\delta \Omega),
\end{equation}
\begin{equation}
\label{eq:b6}
\delta \dot{\Omega} = J^{-1}(\delta M + (\widehat{J\Omega_d} - \widehat{\Omega}_dJ)\delta\Omega),
\end{equation}
Finally \eqref{eq:b2a}, \eqref{eq:b2b}, \eqref{eq:b3}, \eqref{eq:b4}, \eqref{eq:b5} \& \eqref{eq:b6}, combined together results in the linearized system dynamics given in \eqref{eq:linearEq} and \eqref{eq:linearEqforA}. 

The constraint in the variation on $S^2$  given by, $\xi_i.q_{id} = 0$ results in a variation based constraint for the linearized dynamics. Constraint is valid for all time thus, $ \frac{d}{dt}(\xi_i.q_{id}) = 0$ and this gives, $-\omega^T_{id}\hat{q}_{id}\xi + q^T_{id}(\delta \omega) = 0$. These two constraints  applied for all links results in the constraint martrix given in \eqref{eq:linearEqConst}.



\balance	
\bibliographystyle{IEEEtranS}
\bibliography{references,intro_refs}

\end{document}